\documentclass[11pt,draft]{amsart}
\usepackage[utf8]{inputenc}
\usepackage[T1]{fontenc}
\usepackage[letterpaper,centering]{geometry}
\usepackage{lmodern}
\usepackage{eucal}
\usepackage{mathrsfs}
\usepackage[shortlabels]{enumitem}
\usepackage{amsmath,amssymb,amsfonts}
\usepackage{xcolor}
\definecolor{darkblue}{rgb}{0,0,0.3}
\definecolor{darkgreen}{rgb}{0,0.4,0}
\usepackage[ps,arrow,matrix,tips,line]{xy}
{\setbox0\hbox{$ $}}\fontdimen16\textfont2=\fontdimen17\textfont2
\entrymodifiers={+!!<0pt,\the\fontdimen22\textfont2>}
\SelectTips{cm}{11}
\setlist[enumerate]{label={\upshape(\arabic*)},partopsep=.5ex,leftmargin=*}
\setlist[itemize]{leftmargin=*}

\usepackage[colorlinks=true,linkcolor=darkblue, citecolor=darkgreen, unicode,pdfborder={0 0 0},final]{hyperref}

\usepackage{ifdraft}
\ifdraft{

}{}

\theoremstyle{plain}
\newtheorem{thm}{Theorem}[section]

\newtheorem{quest}[thm]{Question}
\newtheorem{conj}[thm]{Conjecture}

\newtheorem{lem}[thm]{Lemma}
\newtheorem{cor}[thm]{Corollary}
\newtheorem{prop}[thm]{Proposition}
\theoremstyle{definition}

\newtheorem{rmk}[thm]{Remark}

\numberwithin{equation}{section}
\theoremstyle{remark}

\newcommand{\bC}{{\mathbb C}}
\newcommand{\R}{{\mathbb R}}
\newcommand{\Q}{{\mathbb Q}}
\newcommand{\Z}{{\mathbb Z}}

\newcommand{\alg}{\mathrm{alg}}
\newcommand{\Pic}{\mathrm{Pic}}
\newcommand{\BU}{\mathrm{BU}}
\newcommand{\colim}{\mathrm{colim}}

\newcommand{\Hom}{\mathrm{Hom}}

\newcommand{\ci}{\mathcal{C}^{\infty}}

\title{On the smoothability problem with rational coefficients}

\author{Olivier Benoist}
\address{D\'epartement de math\'ematiques et applications, \'Ecole normale sup\'erieure et CNRS, 45~rue d'Ulm, 75230 Paris Cedex 05, France}
\email{olivier.benoist@ens.fr}

\author{Claire Voisin}
\address{Institut de Math\'ematiques de Jussieu-Paris rive gauche et CNRS, 4 Place Jussieu,  75252  Paris Cedex 05, France}
\email{claire.voisin@imj-prg.fr}

\begin{document}

\begin{abstract}
We consider the problem of smoothing algebraic cycles with rational coefficients on smooth projective complex varieties up to homological equivalence.  We show that a solution to this problem would be incompatible with the validity of the Hartshorne conjecture on complete intersections in projective space.
We also solve unconditionally a symplectic variant of this problem.
\end{abstract}

\maketitle

\section*{Introduction}

\subsection{Smoothing cycles}

Let $X$ be a smooth projective variety of dimension $n$ over $\bC$. For~$c\geq 0$, consider the subgroup $H^{2c}(X,\Z)_{\alg}$ of $H^{2c}(X,\Z)$ generated by fundamental classes of (closed irreducible algebraic) subvarieties of codimension $c$ in $X$.
Borel and Haefliger have asked in \cite[Section 5.17]{BH} whether $H^{2c}(X,\Z)_{\alg}$ is generated by classes of \textit{smooth} subvarieties.
In other words, is it possible to smooth algebraic cycles up to homological equivalence?

The strongest result in that direction, due to Koll\'ar and Voisin \cite[Theorem~1.2]{KV} (following an earlier work of Hironaka \cite[Theorem, Section 5, p.~50]{Hironaka}), states that the answer to this question is positive in the range $c>\frac{n}{2}$.  Additional positive results have been obtained by Kleiman \cite[Theorem 5.8]{Kleiman} when $c=2$ and $n\in\{4,5\}$.

It was however discovered by Hartshorne, Rees and Thomas \cite[Theorem 1]{HRT} that the answer to the question of Borel and Haefliger is negative in general.  Further counterexamples have appeared in \cite[Th\'eor\`eme~6]{Debarre},  in \cite[Theorem 0.3]{Benoist} (which includes examples in the threshold case~$c=\frac{n}{2}$), and in \cite[Theorem 1.2]{BD}.

\subsection{Rational coefficients}

We shall focus on the rational analogue of this question.

\begin{quest}
\label{qBHQ}
Let $X$ be a smooth projective variety over $\bC$.  Fix $c\geq 0$. Is the $\Q$-vector space $H^{2c}(X,\Q)_{\alg}$ generated by classes of smooth subvarieties of codimension $c$ in~$X$?
\end{quest}

The above question was investigated by Kleiman in \cite{Kleiman} by exploiting the formula ${c_c(\mathcal{O}_Z)=(-1)^c (c-1)!\,[Z]}$ for the $c$-th Chern class of the structure sheaf of a codimension $c$ subvariety $Z\subset X$ (a well-known consequence of the Grothendieck--Riemann--Roch theorem).  Using this formula, one can reduce Question~\ref{qBHQ} to the case of Chern classes of vector bundles,  which come by pull-back from Grassmannians (see the discussion in \cite[p.~282]{Kleiman}).
On the one hand, this reduces Question \ref{qBHQ} to the case where $X$ is a Grassmannian.
On the other hand, combined with an analysis of the singularities of the Schubert subvarieties of Grassmannians,  this line of reasoning allowed Kleiman to show that Question~\ref{qBHQ} has a positive answer when $c>\frac{n}{2}-1$ (see \cite[Theorem 5.8]{Kleiman}).

In constrast with the integral case, no counterexample to Question \ref{qBHQ} has been discovered to date.  It is not even known whether the original counterexample of Hartshorne, Rees and Thomas to the integral question -- the second Chern class of the tautological bundle on the Grassmannian $G(3,6)$~-- is a $\Q$-linear combination of classes of smooth subvarieties.

\subsection{Relation with Hartshorne's conjecture}

In the influential article \cite{hartshorne}, Hart\-shorne introduced the following conjecture (in a more precise form: there, the explicit bound $n>3c$ is suggested).

\begin{conj}
\label{conjH}
Let $X$ be a smooth subvariety of codimension $c$ in $\mathbb{P}^n$.
If $n\gg c$, then $X$ is a complete intersection.
\end{conj}


Our first main result, proven in Section \ref{sec2},  is as follows.

\begin{thm}
\label{th2}
Question \ref{qBHQ} and Conjecture \ref{conjH} cannot both have positive answers (even for $c=2$).
\end{thm}

As both Question \ref{qBHQ} and Conjecture \ref{conjH} are in some sense tameness statements concerning algebraic cycles,
we find it surprising that they cannot hold simultaneously.

More precisely, we will show that if Conjecture \ref{conjH} holds for $c=2$,
then Question~\ref{qBHQ} fails to hold for codimension $2$ cycles on $X=G(k,n)$ whenever~$k,n-k\gg 0$ (see Corollary \ref{corodhart}).
Our strategy of proof is as follows.
A smooth codimension $2$ subvariety~$Z\subset G(k,n)$ is the zero locus of a section of a rank~$2$ vector bundle $E$ on $G(k,n)$ (by Serre's construction and a Barth\nobreakdash--Lefschetz-type theorem).  We embed both $\mathbb{P}^k$ and~$\mathbb{P}^{n-k}$ in~$G(k,n)$, in many different ways. According to an equivalent formulation of
Conjecture~\ref{conjH} for~$c=2$ (see Conjecture~\ref{conjhartshorne} below),  the restrictions of $E$ to all these copies of projective space should split as direct sums of line bundles. The heart of the proof, to which \S\ref{heart} is devoted, is to deduce that $E$ itself splits as a direct sum of line bundles.  This information constrains~$c_2(E)$, hence the cycle class of~$Z$, and this contradicts the validity of Question~\ref{qBHQ}.

The important particular case of unstable vector bundles is considered in \S\ref{parunstable}.
Weaker results concerning the splitting of vector bundles of rank $\geq 3$ as direct sums of line bundles, obtained following the same strategy, appear in \S\ref{parhigher}.

\subsection{A symplectic analogue}

Our second main result answers positively the symplectic counterpart of Question \ref{qBHQ}.

\begin{thm}
\label{th1}
Let $(M,\omega)$ be a compact symplectic $\ci$ manifold.
For all~$c\geq 0$, the $\Q$\nobreakdash-vector space $H^{2c}(M,\Q)$ is generated by fundamental classes of symplectic~$\ci$ submanifolds of $M$ of (real) codimension $2c$.
\end{thm}

The proof of Theorem \ref{th1} is given in Section \ref{sec1}.  We construct in \S\ref{parSullivan} topological complex vector bundles on $M$ by means of Sullivan's rational homotopy theory, and we apply in \S\ref{parAD} the Auroux--Donaldson theorem to find sections of these bundles whose zero loci are symplectic submanifolds.

\vspace{.5em}

We claim that Theorem \ref{th1} shows that there are no purely topological obstructions to the validity of Question \ref{qBHQ}, in a very strong sense (and therefore, in the light of Theorem~\ref{th2},  that it can be viewed as a negative indication concerning the Hartshorne conjecture).  To justify this claim, consider Question~\ref{qBHQ} in the crucial case (as we explained above) where~$X=G(k,n)$ is a Grassmannian.

 As $H^*(X,\Z)$ is then entirely algebraic,  Question \ref{qBHQ} predicts that $H^*(X,\Q)$ is generated as a $\Q$-vector space by classes of smooth algebraic subvarieties.
For this to be true, it is necessary that  $H^*(X,\Q)$ be generated by classes of orientable $\ci$ submanifolds. With integral coefficients, this is not always true, and this was the original topological obstruction to the question of Borel and Haefliger exploited by Harthorne, Rees and Thomas in~\cite{HRT}.  In contrast,  this necessary condition always  holds with rational coefficients,  as was shown by Thom \cite[Corollaire II.30]{Thom}.

One can devise finer possible topological obstructions to the validity of Question~\ref{qBHQ}.  If $H^*(X,\Q)$ were generated by classes of smooth algebraic subvarieties,  then it would a fortiori be generated by
classes of orientable $\ci$ submanifolds $Y\subset X$ whose normal bundle~$N_{Y/X}$ has a complex structure (with integral coefficients,  this obstruction is considered in \mbox{\cite[\S3]{HRT}}),  whose tangent bundle $T_Y$ also has a complex structure,  and such that moreover there is an isomorphism of complex topological vector bundles $N_{Y/X}\oplus T_Y\simeq T_X|_Y$.

Theorem \ref{th1} dashes any hope that such topological obstructions might lead to a counterexample to Question \ref{qBHQ}.  To see it, fix a (K\"ahler) symplectic form $\omega$ on $X$, and let~$Y\subset X$ be a symplectic $\ci$ submanifold.  Identify $N_{Y/X}$ with the $\omega$-orthogonal complement of $T_Y$ in $TX|_Y$.
 Choose $\omega$-compatible complex structures on $N_{Y/X}$ and $T_Y$ (they exist by \cite[Proposition~2.6.4~(i)]{MS}). The uniqueness of $\omega$\nobreakdash-compatible complex structures up to homotopy (apply \cite[Proposition 2.6.4~(i)]{MS} again) shows that $N_{Y/X}\oplus T_Y$ and $T_X|_Y$ are isomorphic as topological complex vector bundles.

\section{Relation between the Borel--Haefliger question and the Hartshorne conjecture}
\label{sec2}

We work over the field $\bC$ of complex numbers. Let $G(k,V)$ be the Grassmanniann of vector subspaces of dimension $k$ of a complex vector space $V$, and set $G(k,n):=G(k,\bC^n)$.
A vector bundle on an algebraic variety is said to be \textit{decomposable} if it is a direct sum of line bundles.

\subsection{Statement}

Thanks to the  Barth--Lefschetz theorem  \cite{barth1, Larsen}  and  Serre's construction, the particular case $c=2$ of Hartshorne's Conjecture \ref{conjH} is equivalent to the following conjecture (see \cite[Conjecture 6.3]{hartshorne} and the discussion surrounding it).

\begin{conj}
\label{conjhartshorne}
There exists an integer $n_0\geq 5$ such that any rank $2$ vector bundle on~$\mathbb{P}^n$,  with $n\geq n_0$,  is decomposable.
\end{conj}

Hartshorne's original conjecture predicts that one can choose $n_0=7$, but no counterexample for $n_0=5$ is known.
We prove the following implication.

\begin{prop}
\label{propcons}
Assuming Conjecture \ref{conjhartshorne},  if $k,n-k\geq n_0$, any rank $2$  vector bundle on~$G(k,n)$ is decomposable.
\end{prop}

Let  $\mathcal{L}$  be the  Pl\"{u}cker line bundle on  $G(k,n)$.

\begin{cor}
\label{corodhart}
If Conjecture \ref{conjhartshorne} holds true, any smooth codimension $2$ subvariety~$Z$ of the Grassmannian $G(k,n)$ with $k,n-k\geq n_0$ is a complete intersection and its  cohomology class $[Z]$ is thus   proportional to  $c_1(\mathcal{L})^2$. In particular,   classes of smooth codimension~$2$ subvarieties do not generate over  $\mathbb{Q}$  the space $H^4(G(k,n), \mathbb{Q})$.
\end{cor}

\begin{proof}
Since $n_0\geq 5$, one has $n\geq 10$.  A Barth--Lefschetz-type theorem due to Sommese (see \cite[Proposition 3.4 and (3.6.3)]{Sommese}) therefore shows that if $Z$ is as in the corollary,  then the restriction maps $H^l(G(k,n),\Z)\to H^l(Z,\Z)$ are bijective for $l\leq 2$, and hence that~$\Pic(Z)$ is generated by  $\mathcal{L}_{\mid Z}$. In  particular  $Z$ is subcanonical and, as~$H^2(G(k,n), \mathcal{L}^{i})=0$ for all~$i$, the Serre construction  applies, providing a  rank 2 vector bundle $E$ such that    $Z$ is the  zero-locus of a  section of $E$. The  class of  $Z$ is then equal to~$c_2(E)$ so the corollary follows from Proposition \ref{propcons}.
\end{proof}

\subsection{Proof of Proposition \ref{propcons}}
\label{heart}

The  Grassmannian $G(k,n)$ is swept out by  $\mathbb{P}^k$'s and by~$\mathbb{P}^{n-k}$'s. We get a  $\mathbb{P}^k\subset G(k,n)$ by  fixing a vector subspace  $W\subset \bC^n$ of  dimension~$k+1$ and by  considering the set of its   hyperplanes; we will denote it by  $\mathbb{P}^k_W$. We next  get  a~$\mathbb{P}^{n-k}\subset G(k,n)$ by  fixing  a  vector subspace $W'\subset \bC^n$ of  dimension $k-1$ and by considering the set of  vector subspaces of  $\bC^n$ of  dimension $k$ containing  $W'$; we will denote it by  $\mathbb{P}^{n-k}_{W'}$.
We note that the intersection~$\mathbb{P}^k_W\cap \mathbb{P}^{n-k}_{W'}$ is a line when   $W'\subset W$.

Let    $k,n-k\geq n_0$ and let $E$ be a rank $2$  vector bundle on  $G(k,n)$.
For any  vector subspace  $W\subset \bC^n$  of  dimension~$k+1$, Conjecture~\ref{conjhartshorne} gives
\begin{eqnarray}
\label{eqrestW}
E_{\mid \mathbb{P}^k_W}=\mathcal{O}(a_W)\oplus \mathcal{O}(b_W),
\end{eqnarray}
for some integers  $a_W,\,b_W$ a priori depending   on  $W$.
Similarly, for any vector subspace $W'\subset \bC^n$ of  dimension $k-1$, Conjecture \ref{conjhartshorne} gives
\begin{eqnarray}
\label{eqrestWprime}
E_{\mid \mathbb{P}^{n-k}_{W'}}=\mathcal{O}(a'_{W'})\oplus \mathcal{O}(b'_{W'}),
\end{eqnarray}
for some integers  $a'_{W'},\,b'_{W'}$ a priori   depending of  $W'$.

\begin{lem}
\label{letruc}
The pair  $\{a_W,b_W\}$ is in fact independent of  $W$ and we have $\{a_W,b_W\}=\{a_{W'},b_{W'}\}$.
\end{lem}

\begin{proof}
 If  $W'\subset W$, the intersection $\mathbb{P}^k_W\cap \mathbb{P}^{n-k}_{W'}$ is a line  $\Delta\cong\mathbb{P}^1$, and by restricting  (\ref{eqrestW}) and (\ref{eqrestWprime}) to  $\Delta$, we  conclude that
\begin{eqnarray}
\label{eqegaliteab}
\{a_W,b_W\}=\{a'_{W'},b'_{W'}\}.
\end{eqnarray}
If  $W_1$ and $W_2$ are two vector  subspaces of  $\bC^n$  of  dimension $k+1$ with ${\rm dim}\,W_1\cap W_2=k-1$, by applying  (\ref{eqegaliteab}) to both  inclusions $W_1\cap W_2\subset W_1$ and $W_1\cap W_2\subset W_2$, we  conclude that
\begin{eqnarray}
\label{egalitequandinter}
\{a_{W_1},b_{W_1}\}=\{a_{W_2},b_{W_2}\}.
\end{eqnarray}
We finally  apply the following   elementary result.

\begin{lem}\label{lemmefacile} Let  $W_0,\,W_\infty\subset \bC^n$ be two vector  subspaces  of  dimension $k+1$. There exist vector  subspaces   $W_i,\,\,1\leq i \leq N$  of  $\bC^n$ of  dimension $k+1$  such that
$${\rm dim}\, W_0\cap W_1=k-1,\,\dots,\,{\rm dim}\,W_i\cap W_{i+1}=k-1,\,\,\dots,\,{\rm dim}\,W_{N}\cap W_\infty=k-1.$$
\end{lem}

Lemma \ref{lemmefacile}  concludes the proof of Lemma \ref{letruc}   by iterated applications of  (\ref{egalitequandinter}).
\end{proof}
\begin{rmk}\label{remapourplustard} {\rm The proof of Lemma \ref{letruc} shows more generally that any rank $r$ vector bundle on $G(k,n)$ that is split on each $\mathbb{P}^k_W$ and $ \mathbb{P}^{n-k}_{W'}$ has constant splitting type and the splitting type is the same for $\mathbb{P}^k_W$ and $ \mathbb{P}^{n-k}_{W'}$. If $r\leq k$, there is an even easier proof of the first  statement. Namely, the constancy follows from the fact that, for a split vector bundle of rank $r$ on $\mathbb{P}^k$ with $r\leq k$, its splitting type is determined by its Chern classes.}
\end{rmk}
From now on, we denote  by  $\{a,b\}$ with  $a\geq b$ the pair $\{a_W,b_W\}$. We now distinguish  two cases.

\vspace{0.5cm}

1) {\bf Case where  $a>b$.}
Consider the  correspondence
\begin{equation}
\label{eqdiag1}
\begin{aligned}
\xymatrix
@R=0.3cm
{
P \ar^{\hspace{-1em}q}[r]\ar_{p}[d]& G(k,n) \\
\hspace{-2em}G(k+1,n)\hspace{-2em}
}
\end{aligned}
\end{equation}
given by the universal family  of $\mathbb{P}^k_W$'s,
where  $p$ is a projective bundle with fiber  $\mathbb{P}^k_W$ over the point  $[W]\in G(k+1,n)$. We denote respectively by
$\mathcal{L}$ and $\mathcal{L}'$  the  Pl\"{u}cker line bundles on~$G(k,n)$ and $G(k+1,n)$.
As we have  $h^0(E_{\mid \mathbb{P}^k_W}(-a))=1$ for any  $[W]\in G(k+1,n)$, the sheaf~$R^0p_*(q^*(E\otimes\mathcal{L}^{-a}))$ is a line bundle on  $G(k+1,n)$, hence  isomorphic  to
${\mathcal{L}'}^{\otimes l}$ for some integer  $l$. Furthermore, the natural inclusion
$$p^*{\mathcal{L}'}^{\otimes l}\subset q^*(E\otimes\mathcal{L}^{-a})$$
is the inclusion of a line subbundle, and the quotient is a line bundle $\mathcal{M}$ on  $P$,
which
fits into an
 exact sequence
\begin{eqnarray}
\label{eqsuiteexacte}
0\rightarrow p^*{\mathcal{L}'}^{\otimes l}\rightarrow q^*(E\otimes\mathcal{L}^{-a})\rightarrow \mathcal{M}
\rightarrow0.
\end{eqnarray}
Computing determinants, we conclude that
$$ \mathcal{M}
= q^*({\rm det}\,(E\otimes \mathcal{L}^{-a}))\otimes p^*{\mathcal{L}'}^{\otimes -l}.$$
Furthermore  (\ref{eqsuiteexacte}) also gives

\begin{eqnarray}
\label{eqsuiteexactec2}
q^*c_2(E\otimes\mathcal{L}^{-a})=p^*c_1({\mathcal{L}'}^{\otimes l})\cdot(q^*c_1(E\otimes \mathcal{L}^{-a})-p^*c_1({\mathcal{L}'}^{\otimes l}))\,\,{\rm in}\,\,H^4(P,\mathbb{Z}).
\end{eqnarray}
Restricting  (\ref{eqsuiteexactec2}) to the fibers of  $q$, which are  projective spaces  of dimension $>1$, we conclude that  $l=0$, so that
$$ \mathcal{M}
= q^*({\rm det}\,(E\otimes \mathcal{L}^{-a}))
,$$
and (\ref{eqsuiteexacte}) writes
\begin{eqnarray}
\label{eqsuiteexacte2}
0\rightarrow \mathcal{O}_P\rightarrow q^*(E\otimes\mathcal{L}^{-a})\rightarrow q^*({\rm det}\,(E\otimes \mathcal{L}^{-a}))
\rightarrow0.
\end{eqnarray}
  Next,  applying  $R^0q_*$ to the  exact sequence  (\ref{eqsuiteexacte2}), we get that  $E$ is  an extension of two line bundles   on  $G(k,n)$, hence is decomposable.

\vspace{0.5cm}

2) {\bf Case where   $a=b$.} In this  case, the bundle $E\otimes \mathcal{L}^{-a}$ is trivial on the $\mathbb{P}^{k}_{W}$'s
and in particular we have
\begin{eqnarray}
\label{eqtrivialsurfibre}
q^*(E\otimes\mathcal{L}^{-a})\cong p^*F,
\end{eqnarray}
for some vector bundle  $F$ of  rank $2$  on  $G(k+1,n)$. The  Grassmannian   $G(k+1,n)$ is itself swept out by  projective spaces  $\mathbb{P}^{k+1}_{W''}$ of  dimension $k+1$, associated to vector subspaces~$W''\subset \bC^n$ of dimension $k+2$, and Conjecture \ref{conjhartshorne} implies that
the  restriction~$F_{\mid  \mathbb{P}^{k+1}_{W''}}$ is  decomposable, that is
\begin{eqnarray}
\label{eqnoudu24}
F_{\mid  \mathbb{P}^{k+1}_{W''}}=\mathcal{O}_{  \mathbb{P}^{k+1}_{W''}}(a')\oplus  \mathcal{O}_{  \mathbb{P}^{k+1}_{W''}}(b').
\end{eqnarray}
For any   pair  $(W''',W'')$ of  vector subspaces of $\bC^n$ of respective dimensions $k$ and $k+2$, there is a line
$\Delta_{W''',W''}\subset G(k+1,n)$ parameterizing  the vector subspaces   $V_t$ of dimension~$k+1$ of  $\bC^n$ containing $W'''$ and  contained in
$W''$. With the same notation as in (\ref{eqdiag1}), this line~$\Delta_{W''',W''}$   admits a  canonical lift $\widetilde{\Delta}_{W''',W''}\subset P$ given by  the  section of  $p$
$$ [V_t]\mapsto ([W'''], [V_t])\in P,$$
  which is defined on  $\Delta_{W''',W''}$.
Restricting equality  (\ref{eqtrivialsurfibre})  to   the line $\widetilde{\Delta}_{W''',W''}\subset P$ which is  contained in a   a fiber of   $q$, we get that  $F_{\mid  \Delta_{W''',W''}}$ is  trivial,  and by applying  (\ref{eqnoudu24}),  we conclude that  $F_{\mid \mathbb{P}^{k+1}_{W''}}$ is  trivial, since  $ \Delta_{W''',W''}$ is contained in  $  \mathbb{P}^{k+1}_{W''}$. In other words, considering  the incidence correspondence
\begin{equation*}
\begin{aligned}
\xymatrix
@R=0.3cm
{
P' \ar^{\hspace{-1.5em}q'}[r]\ar_{p'}[d]& G(k+1,n) \\
\hspace{-2em}G(k+2,n)\hspace{-2em}
}
\end{aligned}
\end{equation*}
for the family
of $\mathbb{P}^{k+1}_{W''}$ covering $G(k+1,n)$,
we proved that
$${q'}^* F={p'}^* G,$$ for some  rank $2$ bundle  $G$ on   $G(k+2,n)$. The  triviality of  $E(-a)$, which is   implied by the triviality of  $F$,  is thus implied  by that of  $G$.
Iterating the above reasoning and  using the  fact that  $G(n-1,n)=\mathbb{P}^{n-1}$, we  finally conclude that  $E(-a)$ is  trivial, so Proposition \ref{propcons} is also proved in this case.

\subsection{The unstable variant}
\label{parunstable}

An interesting  and very  plausible variant of Hartshorne's Conjecture \ref{conjhartshorne} is the following statement.

\begin{conj}
\label{conjhartshorneunstable}
There exists $n_0'\geq 5$ such that any unstable  rank $2$ vector bundle $E$ on~$\mathbb{P}^n$,  with $n\geq n_0'$,  is decomposable.
\end{conj}

In this statement, ``unstable'' means ``not slope stable'', and this condition  says that  there exists a nonzero section of $E(-a)$ for some $a$ such that  ${\rm det}\,E(-a)\leq 0$.
Conjecture~\ref{conjhartshorneunstable} is important and, as explained in \cite{schneider}, it  would have  remarkable consequences, such as the existence of topological complex vector bundles (of rank $2$) on projective space that do not have an algebraic structure. This conjecture  is studied   in \cite{grauertschneider} and \cite{siurank2} but the proofs there have serious  gaps, hence it is still open.

Let $G(k,n)$ be a Grassmannian as in the previous sections. As ${\rm Pic}(G(k,n))=\mathbb{Z} \mathcal{L}$, we can speak as above of unstable rank $2$ vector bundles $F$ on $G(k,n)$.
We make the following observation:

\begin{lem}
\label{lemunstable}
Let $F$ be a unstable rank $2$ vector bundle on $G(k,n)$, with $k,n-k\geq 1$. Then for any ${k+1}$\nobreakdash-di\-men\-sion\-al vector subspace $W\subset \mathbb{C}^n$, and any  $k-1$\nobreakdash-dimensional vector subspace $W'\subset \mathbb{C}^n$, the restricted vector bundles
$$F_{\mid \mathbb{P}^k_W},\,\,  F_{\mid \mathbb{P}^{n-k}_{W'}}$$
are unstable.
\end{lem}

\begin{proof}
We use the notation $F(-a):=F\otimes \mathcal{L}^{\otimes -a}$. By assumption, there exists a twist $F(-a)$ of $F$ which admits a nonzero section $s$, and is such that ${\rm det}\,F(-a)\leq 0$.
If $W$ and $W'$ are general, the restrictions of $s$ to $\mathbb{P}^k_W$ and  $\mathbb{P}^{n-k}_{W'}$ are nonzero, implying  that  the restricted vector bundles $F_{\mid \mathbb{P}^k_W},\,\,  F_{\mid \mathbb{P}^{n-k}_{W'}}$ are unstable.
By upper semi-continuity, the fact that $F_{\mid \mathbb{P}^k_W}(-a)$ has a nonzero section for general $W$ implies the same property for any~$W$, and similarly for the  $W'$'s.  Hence we conclude that $F_{\mid \mathbb{P}^k_W},\,\,  F_{\mid \mathbb{P}^{n-k}_{W'}}$ are unstable for any~$W,\,W'$.
\end{proof}

Using Lemma \ref{lemunstable}, the same arguments as in the previous section then give us the following unstable variant of Proposition \ref{propcons}.

\begin{cor}
\label{corounstablesplit}
If Conjecture \ref{conjhartshorneunstable} is true, then any unstable vector bundle on $G(k,n)$, with~$k,n-k\geq n_0'$, is decomposable, hence its second Chern class is a multiple of $c_1(\mathcal{L})^2$.
\end{cor}

\subsection{Extension to higher rank bundles}
\label{parhigher}

We finally  prove using  similar  arguments  an  analogous but  weaker statement for bundles of arbitrary rank $r\geq 2$ on $G(k,n)$. The next question, in the spirit of Hartshorne's conjectures \ref{conjH} and \ref{conjhartshorne},  was not stated as a con\-jec\-ture by Hartshorne for lack of sufficient evidence (see the end of \cite[\S6]{hartshorne}).

\begin{quest}
\label{conjhartshornebis}
Fix $r\geq 2$.  Does there exist an integer $n_0(r)\geq 5$ such that any rank $r$ vector bundle $E$ on~$\mathbb{P}^n$,  with $n\geq n_0(r)$,  is decomposable?
\end{quest}

\begin{prop}
\label{propconsr}
If  Question \ref{conjhartshornebis} has a positive answer for $r\geq 2$, then for any rank $r$  vector bundle $E$ on $G(k,n)$ with $n-k\geq k\geq n_0(r)$, the Chern classes $c_i(E)$ for  $i=1,\ldots, r$  satisfy
\begin{eqnarray}
\label{eqchernrest}
\int_{G(k,n)}c_i(E)l^{k-i} (\gamma_k-l^{n-2k}\gamma_{n-k})=0,
\end{eqnarray}
where $l=c_1(\mathcal{L})$  and  $$\gamma_k\in H^{2N-2k}(G(k,n),\mathbb{Z}),\,\,{\rm  resp.}\,\, \gamma_{n-k}\in H^{2N-2n+2k}(G(k,n),\mathbb{Z}),\,\,N:={\rm dim}\,G(k,n)$$
 is the  class of any $\mathbb{P}^k_W$, resp.  $\mathbb{P}^{n-k}_{W'}$.
\end{prop}

\begin{proof}
The notation is as in \S\ref{heart}. Under the assumptions $$n-k\geq k\geq n_0(r),$$
a positive answer to Question \ref{conjhartshornebis} implies that the  restrictions
$E_{\mid \mathbb{P}^k_W}$ and $E_{\mid \mathbb{P}^{n-k}_{W'}}$  are  decomposable  vector bundles and by the same arguments as in the proof of Lemma \ref{letruc} (see  Remark \ref{remapourplustard}),  we conclude that  the  type of the  decomposition
is the same for   $E_{\mid \mathbb{P}^k_W}$ and  $ E_{\mid \mathbb{P}^{n-k}_{W'}}$ (and in particular it does not depend on the choice of  $W$ and  $W'$).
Hence there  exist  integers  $a_1,\ldots,\,a_r$  such that
\begin{eqnarray}
\label{eqpourrestr}
\hspace{3em} E_{\mid \mathbb{P}^k_W}
=\mathcal{O}_{\mathbb{P}^k_W}(a_1)\oplus\ldots\oplus  \mathcal{O}_{\mathbb{P}^k_W}(a_r) \textrm{ \,\,and\,\, } E_{\mid \mathbb{P}^{n-k}_{W'}}=\mathcal{O}_{\mathbb{P}^{n-k}_{W'}}(a_1)\oplus\ldots\oplus  \mathcal{O}_{\mathbb{P}^{n-k}_{W'}}(a_r).
\end{eqnarray}
It follows from  (\ref{eqpourrestr}) that
\begin{eqnarray}
\label{eqdevlci}
c_i(E)_{\mid \mathbb{P}^k_W}=\sigma_i(a_j)h^i,\,\,c_i(E)_{\mid\mathbb{P}^{n-k}_{W'}}=\sigma_i(a_j){h'}^i,
\end{eqnarray}
where  $h=c_1(\mathcal{O}_{ \mathbb{P}^k_W}(1))=l_{\mid  \mathbb{P}^k_W},\,\,\,h'=c_1(\mathcal{O}_{ \mathbb{P}^{n-k}_{W'}}(1))=l_{\mid  \mathbb{P}^{n-k}_{W'}}$ and   $\sigma_i$  is the  $i$-th symmetric function  of  $r$ variables.
It follows from  (\ref{eqdevlci}) that
\begin{eqnarray}
\label{eqdevlciintegre}
\int_{G(k,n)} l^{k-i}c_i(E)\gamma_k=\int_{\mathbb{P}^k_W} h^{k-i}c_i(E)=\sigma_i(a_j)\\
\nonumber
\int_{G(k,n)} l^{n-k-i}c_i(E)\gamma_{n-k}=\int_{\mathbb{P}^{n-k}_{W'}} (h')^{n-k-i}c_i(E)=\sigma_i(a_j),
\end{eqnarray}
which implies    (\ref{eqchernrest}).
\end{proof}

The following lemma shows that the equations (\ref{eqchernrest}) impose $r-1$ independent equations  to the  Chern classes of  $E$ which are thus severely restricted if Question \ref{conjhartshornebis} has a positive answer.

\begin{lem}
Under the hypotheses of Proposition \ref{propconsr}, the morphism
$$H^{2i}(G(k,n),\mathbb{Z})\rightarrow  H^{2N}(G(k,n),\mathbb{Z}) $$
of  cup-product  by the class $l^{k-i}(\gamma_k-l^{n-2k}\gamma_{n-k})$
is nonzero for  $i\geq 2$. Equivalently, thanks to  Poincar\'{e} duality, the class  $l^{k-i}(\gamma_k-l^{n-2k}\gamma_{n-k})$
is nonzero  for $i\geq 2$.
\end{lem}

\begin{proof}
 It clearly suffices to prove that  the cohomology class  $l^{k-2}(\gamma_k-l^{n-2k}\gamma_{n-k})$ is nonzero in  $H^{2N-4}(G(k,n),\mathbb{Z})$.  This  class is the  difference of the  class  of a  plane in $\mathbb{P}^k_W$  and the  class  of  a  plane in
 $\mathbb{P}^{n-k}_{W'}$. We  construct an inclusion
$$G(2,4)=G(2,V')\subset G(k,n)$$ by choosing a  subspace
$V\subset \bC^n$ of codimension $n-k-2$ and a  quotient $V'$ of  $V$ of dimension~$4$. This inclusion induces  a surjection
$H^4(G(k,n),\mathbb{Z})\rightarrow  H^4(G(2,4),\mathbb{Z})$, as one can see by computing the pull-back to $G(2,4)$ of  the second Chern class of the universal bundle on~$G(k,n)$, and thus an injection
$H_4(G(2,4),\mathbb{Z})\rightarrow  H_4(G(k,n),\mathbb{Z})$. The  Grassmannian~$G(2,4)$ is a quadric in  $\mathbb{P}^5$. It  contains the two  types of  planes considered above and their   classes    are  independent in $H_4(G(2,4),\mathbb{Z})$; hence they remain independent in~$H_4(G(k,n),\mathbb{Z})$.
\end{proof}

\section{Construction of symplectic submanifolds}
\label{sec1}

\subsection{Construction of topological vector bundles}
\label{parSullivan}

Let $X$ be a connected and simply connected $CW$-complex. The \textit{rationalization} of~$X$ in the sense of Sullivan is a continuous map $\rho : X\to X_{\Q}$ of connected and simply connected $CW$-complexes that satisfies the following equivalent properties:
\begin{enumerate}[(i)]
\item for all $i\geq 1$, the abelian group $H_i(X_{\Q},\Z)$ is a $\Q$-vector space and the morphism $\rho_*:H_i(X,\Q)\to H_i(X_{\Q},\Q)$ is an isomorphism;
\item
\label{pi}
for all $i\geq 1$, the abelian group $\pi_i(X_{\Q})$ is a $\Q$-vector space and the morphism ${\rho_*:\pi_i(X)\otimes_{\Z}\Q\to \pi_i(X_{\Q})\otimes_{\Z}\Q}$ is an isomorphism.
\end{enumerate}
 It exists and it is unique up to homotopy (see for instance \cite[Theorem 6.1.2]{Moreconcise}).

Let $M$ be an abelian group and fix $n\geq 2$.
Let $K(M,n)$ denote an Eilenberg\nobreakdash--Maclane $CW$-complex, whose only nontrivial homotopy group is $M$ in degree~$n$.  It is a classifying space for $H^n(-,M)$ (see \cite[Theorem 4.57]{Hatcher}).
The characterization~\ref{pi} shows that  the rationalization of $K(M,n)$ is $K(M\otimes_\Z\Q,n)$.

\begin{prop}
\label{construction}
Let $X$ be a finite $CW$-complex. Fix $r\geq 1$ and, for ${1\leq i\leq r}$, a class $\alpha_i\in H^{2i}(X,\Z)$.  Then there exist $m\geq 1$ and a complex topological vector bundle $E$ of rank~$r$ on $X$ such that $c_i(E)=m\cdot\alpha_i$ in $H^{2i}(X,\Z)$ for all $1\leq i\leq r$.
\end{prop}

\begin{proof}
Let $\BU(r)$ be the classifying space for complex topological vector bundles of rank~$r$, constructed as the increasing union of the complex Grassmannians $G(r, n)$ for all $n\geq r$. Let $\rho:\BU(r)\to\BU(r)_{\Q}$ be its rationalization.

The Chern classes induce a continuous map $(c_1,\dots, c_r):\BU(r)\to\prod_{i=1}^r K(\Z,2i)$.
On the one hand, $H^*(\BU(r),\Z)=\Z[c_1,\dots,c_r]$ (see \cite[Theorem 14.5]{MS}).  On the other hand, $H^*(K(\Z,2i),\Q)$ is a polynomial ring in the tautological element of $H^{2i}(K(\Z,2i),\Q)$ (see \cite[Proposition 4 p.~501]{Serre}).
It therefore follows from the K\" unneth formula that the rationalized map
\begin{equation}
\label{Chernmap}
(c_1,\dots,c_r):\BU(r)_{\Q}\to\prod_{i=1}^r K(\Q,2i)
\end{equation}
induces an isomorphism between integral cohomology groups,  and hence between integral homology groups by the universal coefficient theorem (as these groups are $\Q$-vector spaces).  We deduce from \cite[Theorem B]{May} that~(\ref{Chernmap}) is a homotopy equivalence, and we denote by $\Phi:\prod_{i=1}^r K(\Q,2i)\to \BU(r)_{\Q}$ a homotopy inverse.

Define $Y:=\prod_{i=1}^rK(\Z,2i)$. By cellular approximation, we may assume that the continuous maps $\mu_m:Y\to Y$ induced by multiplication by $m$ on the coefficients~$\Z$ preserve the $d$-skeleton $Y_d$ of $Y$ for all $d\geq 0$.

\begin{lem}
\label{lem}
For all $d\geq 0$,  there exist a continuous map $\Psi_d: Y_d\to \BU(r)$ and an integer~$m_d\geq 1$ such that the following diagram commutes up to homotopy:
\begin{equation}
\label{grosdiagramme}
\begin{aligned}
\xymatrix
@R=0.3cm
{
Y_d\ar@{.>}^{\hspace{.5em}\Psi_d}[ddr]\ar@{^(->}[d]& \\
Y\ar_{\mu_{m_d}}[d]& \\
Y\ar[d]& \ar_{\rho}[d]\BU(r)\\
\prod_{i=1}^rK(\Q,2i)\ar^{\hspace{.9em}\Phi}_{\hspace{.9em}\sim}[r]&\BU(r)_{\Q}.
}
\end{aligned}
\end{equation}
\end{lem}

\begin{proof}[Proof of Lemma \ref{lem}]
We argue by induction on $d$. The base case $d=0$ is obvious as both $\BU(r)$ and $\BU(r)_{\Q}$ are connected.  Assume from now on the existence of a lift ${\Psi_d: Y_d\to \BU(r)}$ of the composition $Y\xrightarrow{\mu_{m_d}} Y\to \prod_{i=1}^rK(\Q,2i)\xrightarrow{\Phi}\BU(r)_{\Q}$ to~$\BU(r)$,  in restriction to $Y_d$.  After possibly modifying it, our goal is to extend this lift to $Y_{d+1}$,  at the expense of precomposing by a well-chosen multiplication map~$\mu_m:Y\to Y$ (we will then set $m_{d+1}:=m\cdot m_d$).

By obstruction theory \cite[Theorem 34.2]{Steenrod},  the obstruction to the existence of such a lift lives in  $H^{d+1}(Y,\pi_d(F))$,  where $F$ is the homotopy fiber of~$\rho$. To conclude, it suffices to show that this obstruction class is killed by pull-back by $\mu_m:Y\to Y$ for some $m\geq 1$.

As $\rho:\BU(r)\to\BU(r)_{\Q}$ is the rationalization of~$\BU(r)$,  and as the homology groups of~$\BU(r)$,  hence also its homotopy groups,  are finitely generated \cite[Proposition 1 p.~491]{Serre},  the exact sequence
$$\pi_{d+1}(\BU(r))\to \pi_{d+1}(\BU(r)_{\Q})\to\pi_d(F)\to\pi_d(\BU(r))\to \pi_d(\BU(r)_{\Q})$$
shows that $\pi_d(F)$ is the direct sum of a finite abelian group and of finitely many copies of~$\Q/\Z$.  In addition, since the homology groups of $Y$ are finitely generated \cite[Corollaire~1 p.~500]{Serre},
one can verify that $H^{d+1}(Y,\Q/\Z)=\colim_{l\geq 1} H^{d+1}(Y,\Z/l)$ by applying the universal coefficient theorem.
Consequently, we only need to show that, for all $l\geq 1$,  any class in $H^{d+1}(Y,\Z/l)$ is killed by pull-back by $\mu_m:Y\to Y$ for some well-chosen $m\geq 1$.

As $\Z/l$ is an injective $\Z/l$\nobreakdash-module (see \cite[Exercise 2.3.1]{Weibel}),  the functor $\Hom(-,\Z/l)$ is exact on the abelian category of $\Z/l$-modules.  It follows that $$H^{d+1}(Y,\Z/l)=\Hom(H_{d+1}(Y,\Z/l),\Z/l).$$  Using that $H_{d+1}(Y,\Z/l)$ is finitely generated,  we see that it remains to verify that any class in $H_{d+1}(Y,\Z/l)$ is killed by push-forward by $\mu_m:Y\to Y$ for some well-chosen $m\geq 1$.
Equivalently, we must prove that any class in $H_{d+1}(Y,\Z/l)$ vanishes in the homology of the
homotopy colimit (constructed as a mapping telescope, see \cite[p.~312]{Hatcher}) of the diagram
\begin{equation}
\label{colim}
Y\xrightarrow{\mu_2}Y\xrightarrow{\mu_3} Y\xrightarrow{\mu_4} Y\xrightarrow{\mu_5}\cdots.
\end{equation}
This homotopy colimit
identifies with the rationalization of $Y$ (to see it, compute its homotopy groups and use characterization~\ref{pi} above). Its integral homology groups are therefore $\Q$-vector spaces,  and we deduce that its homology groups with $\Z/l$ coefficients vanish. This concludes the proof.
\end{proof}

Let us resume the proof of Proposition \ref{construction}.  Let $f:X\to Y=\prod_{i=1}^rK(\Z,2i)$ be a continuous map classifying the classes $\alpha_1,\dots, \alpha_r$.  As $X$ is a finite $CW$-complex,  we may assume by cellular approximation that $f(X)\subset Y_d$ for some well-chosen $d\geq 0$.  Let $\Psi_d:Y_d\to \BU(r)$ and $m_d\geq 1$ be as in Lemma \ref{lem}.  Let $m'\geq 1$ be an integer killing the torsion subgroup of $H^*(X,\Z)$. Set $m:=m'\cdot m_d$.

Let $E'$ and $E$ be the complex topological vector bundles of rank $r$ on $X$ classified by
$\Psi_d\,\circ f: X\to \BU(r)$ and $\Psi_d\,\circ\mu_{m'}\circ f: X\to \BU(r)$ respectively. One has
${c_i(E')=m_d\cdot\alpha_i}$ in~$H^{2i}(X,\Q)$ for $1\leq i\leq r$ by commutativity of (\ref{grosdiagramme}). Our choice of~$m'$ implies that $m'\cdot c_i(E')=m\cdot\alpha_i$ in $H^{2i}(X,\Z)$ for $1\leq i\leq r$. As $c_i(E)=m'\cdot c_i(E')$ in $H^{2i}(X,\Z)$ for~$1\leq i\leq r$, the proof of the proposition is complete.
\end{proof}

\begin{rmk}
One can obtain many variants of Proposition \ref{construction} using the same argument.  For instance,  it is possible to  construct a complex topological vector bundle $E$ of rank~$r$ on~$X$ with the property that $c_i(E)=m^i\cdot\alpha_i$ for $1\leq i\leq r$, by replacing everywhere the map $\mu_m:Y\to Y$ by the map $\lambda_m:Y\to Y$ induced by multiplication by $m^i$ on the $i$-th factor of $Y=\prod_{i=1}^rK(\Z,2i)$.
\end{rmk}

\subsection{The Auroux--Donaldson theorem}
\label{parAD}

The next theorem is Auroux's extension to vector bundles \cite[Corollary~1]{Auroux} of Donaldon's construction of symplectic hypersurfaces \cite{Donaldson} in a symplectic manifold $(M,\omega)$ (see also \cite{Sikorav}).

\begin{thm}[Auroux--Donaldson]
\label{DA}
Let $(M,\omega)$ be a compact symplectic $\ci$ manifold.  Then there exists a complex line bundle $L$ on $M$ such that the following holds. For all complex vector bundles $E$ on $M$, and for all $k\gg0$, there exists a $\ci$ section of~$E\otimes L^{\otimes k}$ which is transversal to $0$ and whose zero locus is a symplectic $\ci$ submanifold of $M$.
\end{thm}

To be precise, the Auroux--Donaldson formulation assumes that the de Rham cohomology class $\frac{[\omega]}{2\pi}$ is in the image of the natural map $H^2(M,\Z)\to H^2(M,\R)$. This restriction is well-known to be immaterial, as indicated in \cite[Corollary 6]{Donaldson} and \cite[Remarque~p.~233]{Sikorav}. For the convenience of the reader, we give a few more details on this argument.

\begin{proof}[Proof of Theorem \ref{DA}]
Fix a Riemannian metric $g$ on $M$.
Let $(\omega_i)_{1\leq i\leq N}$ be a collection of closed~$\ci$ $2$-forms on $M$ whose de Rham cohomology classes form a basis of $H^2(M,\R)$.  For $\underline{t}:=(t_1,\dots,t_N)\in\R^N$ small enough, the $2$-form $\omega_{\underline{t}}:=\omega+\sum_{i=1}^Nt_i\omega_i$ is still symplectic.
For such values of $\underline{t}$,  the almost complex structure $J_{\underline{t}}:=J_{g,\,\omega_{\underline{t}}}$ (in the notation of \cite[Proposition 2.5.6]{McDS}) depends continuously on~$\underline{t}$, and is $\omega_{\underline{t}}$\nobreakdash-compatible in the sense that the pairing $(v,w)\mapsto\omega_{\underline{t}}(v,J_{\underline{t}}(w))$ defines a Riemannian metric on $M$.

For $\underline{t}$ small enough, one has $\omega(v,J_{\underline{t}}(v))>0$ for all nonzero tangent vectors $v$ to~$M$. As the set of $\underline{t}\in\R^N$
such that $\frac{[\omega_{\underline{t}}]}{2\pi}\in H^2(M,\Q)\subset H^2(M,\R)$ is dense in $\R^N$, we may fix from now on a~$\underline{t}$ satisfying both conditions. Set $\omega':=\omega_{\underline{t}}$ and $J':=J_{\underline{t}}$.

As the class $\frac{[\omega']}{2\pi}$ is rational,  some integral multiple of it is the first Chern class of a complex line bundle $L$ on $M$. By \cite[Corollary 1]{Auroux}, there exist $\ci$ sections of~$E\otimes L^{\otimes k}$ for~$k\gg0$ which are asymptotically holomorphic (with respect to $J'$) and transverse to~$0$ in the sense of \cite[Definitions 1 and~2]{Auroux}. Their zero loci $N_k\subset M$ are therefore asymptotically $J'$-holomorphic (see \cite[Proposition 1]{Auroux}) in the sense that the subbundles $TN_k$ and~$J'(TN_k)$ of $TM|_{N_k}$ are very close when $k\gg 0$ (with respect to a fixed metric on the Grassmannian bundle associated with $TM$).  The positivity of $\omega(v,J'(v))$ for all nonzero tangent vectors $v$ to $N_k$ now implies that $\omega|_{N_k}$ is nondegenerate,  so the $\ci$ submanifold~$N_k\subset M$ is symplectic.
\end{proof}

\begin{thm}
\label{th1bis}
Let $(M,\omega)$ be a compact symplectic $\ci$ manifold.
For all $c\geq 0$, the $\Q$\nobreakdash-vector space $H^{2c}(M,\Q)$ is generated by fundamental classes of codimension~$2c$ symplectic~$\ci$ submanifolds of $M$.
\end{thm}

\begin{proof}
We may assume that $c\geq 1$.  Choose $\alpha\in H^{2c}(M,\Q)$. By Proposition \ref{construction},  after possibly replacing $\alpha$ by a positive integral multiple,  there exists a complex vector bundle~$E$ of rank $c$ on $M$ such that $c_i(E)=0$ in $H^{2i}(M,\Q)$ for $1\leq i\leq c-1$, and $c_c(E)=\alpha$ in $H^{2c}(M,\Q)$.
Let $L$ be as in Theorem \ref{DA}.
For all $k\gg0$, one can find a $\ci$ section of~$E\otimes L^{\otimes k}$ which is transversal to $0$ and whose zero locus $N_k\subset M$ is symplectic. One computes that $[N_k]=c_c(E\otimes L^{\otimes k})=\alpha+k^cc_1(L)^c$, so
\phantom{\qedhere}
$$\alpha=\frac{(k+1)^c}{(k+1)^c-k^c}[N_k]-\frac{k^c}{(k+1)^c-k^c}[N_{k+1}].\eqno\qed$$
\end{proof}

\bibliographystyle{myamsalpha}
\bibliography{Hartshorne}
\end{document}